\documentclass[11pt]{amsart}
\usepackage{latexsym}
\usepackage{amssymb}
\usepackage{amscd}
\usepackage{float}
\usepackage{graphicx}
\usepackage{amsfonts}
\usepackage{pb-diagram}
 \usepackage{amsmath,amscd}
\usepackage{setspace}

\newtheorem{thm}{Theorem}

\newtheorem{lemma}{Lemma}
\newtheorem{prop}{Proposition}
\newtheorem{defn}{Definition}
\newtheorem{remark}{Remark}

\newtheorem{ex}{Example}
\newtheorem{nt}{Notation}

\begin{document}

\title[Topological steps toward $\mathcal{S}\left(L(p,1)\right)$ via braids]
  {Topological steps toward the Homflypt skein module of the lens spaces $L(p,1)$ via braids}

\author{Ioannis Diamantis}
\address{ International College Beijing,
China Agricultural University,
No.17 Qinghua East Road, Haidian District,
Beijing, {100083}, P. R. China.}
\email{ioannis.diamantis@hotmail.com}

\author{Sofia Lambropoulou}
\address{ Departament of Mathematics,
National Technical University of Athens,
Zografou campus,
{GR-15780} Athens, Greece.}
\email{sofia@math.ntua.gr}
\urladdr{http://www.math.ntua.gr/~sofia}

\author{Jozef Przytycki}
\address{ Departament of Mathematics,
George Washington University,
Phillips Hall, Room 739,
DC 20052 Washington, U.S.A. and University of Gdansk}
\email{przytyck@gwu.edu}
\urladdr{http://home.gwu.edu/~przytyck/}



\thanks{This research  has been co-financed by the European Union (European Social Fund - ESF) and Greek national funds through the Operational Program ``Education and Lifelong Learning" of the National Strategic Reference Framework (NSRF) - Research Funding Program: THALES: Reinforcement of the interdisciplinary and/or inter-institutional research and innovation. }

\setcounter{section}{-1}

\date{}

\begin{abstract}
In this paper we work toward the Homflypt skein module of the lens spaces $L(p,1)$, $\mathcal{S}(L(p,1))$, using braids. In particular, we establish the connection between $\mathcal{S}({\rm ST})$, the Homflypt skein module of the solid torus ST, and $\mathcal{S}(L(p,1))$ and arrive at an infinite system, whose solution corresponds to the computation of $\mathcal{S}(L(p,1))$. We start from the Lambropoulou invariant $X$ for knots and links in ST, the universal analogue of the Homflypt polynomial in ST, and a new basis, $\Lambda$, of $\mathcal{S}({\rm ST})$ presented in \cite{DL1}. We show that $\mathcal{S}(L(p,1))$ is obtained from $\mathcal{S}({\rm ST})$ by considering relations coming from the performance of braid band moves (bbm) on elements in the basis $\Lambda$, where the braid band moves are performed on any moving strand of each element in $\Lambda$. We do that by proving that the system of equations obtained from diagrams in ST by performing bbm on any moving strand is equivalent to the system obtained if we only consider elements in the basic set $\Lambda$.

\smallbreak

The importance of our approach is that it can shed light to the problem of computing skein modules of arbitrary c.c.o. $3$-manifolds, since any $3$-manifold can be obtained by surgery on $S^3$ along unknotted closed curves. The main difficulty of the problem lies in selecting from the infinitum of band moves some basic ones and solving the infinite system of equations.
\end{abstract}

\maketitle

\section{Introduction}\label{intro}

In this paper we relate the Homflypt skein module of the lens spaces $L(p,1)$, $\mathcal{S}(L(p,1))$, to the Homflypt skein module of the solid torus, $\mathcal{S}({\rm ST})$. This work is part of the PhD thesis of the first author \cite{D} and in \cite{DL3} we develop an algebraic approach to the computation of the Homflypt skein module of the lens space $L(p,1)$ via braids. 

\smallbreak

Skein modules are quotients of free modules over ambient isotopy classes of knots and links in a $3$-manifold by properly chosen skein relations. The skein module of a $3$-manifold $M$ based on the Homflypt skein relation is called the \textit{Homflypt skein module} of $M$, also known as \textit{Conway skein module} and as $3^{rd}$ \textit{skein module} (\cite{P, P2}). More precisely, let $M$ be an oriented $3$-manifold, $R=\mathbb{Z}[u^{\pm1},z^{\pm1}]$, $\mathcal{L}$ the set of all ambient isotopy classes of oriented links in $M$ and
let $S$ the submodule of $R\mathcal{L}$ generated by the skein expressions $u^{-1}L_{+}-uL_{-}-zL_{0}$, where $L_{+}$, $L_{-}$ and $L_{0}$ comprise a Conway triple represented schematically by the illustrations in Figure~\ref{skein}.

\begin{figure}[!ht]
\begin{center}
\includegraphics[width=1.7in]{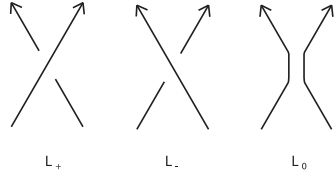}
\end{center}
\caption{The links $L_{+}, L_{-}, L_{0}$ locally.}
\label{skein}
\end{figure}

\noindent For convenience we allow the empty knot, $\emptyset$, and we add the relation $u^{-1} \emptyset -u\emptyset =zT_{1}$, where $T_{1}$ denotes the trivial
knot. Then the {\it Homflypt skein module} of $M$ is defined as:

\begin{equation*}
\mathcal{S} \left(M\right)=\mathcal{S} \left(M;{\mathbb Z}\left[u^{\pm 1} ,z^{\pm 1} \right],u^{-1} L_{+} -uL_{-} -zL{}_{0} \right)={\raise0.7ex\hbox{$
R\mathcal{L} $}\!\mathord{\left/ {\vphantom {R\mathcal{L} S }} \right. \kern-\nulldelimiterspace}\!\lower0.7ex\hbox{$ S  $}}.
\end{equation*}

\bigbreak

Skein modules of $3$-manifolds have become very important algebraic tools in the study of $3$-manifolds, since their properties renders topological information about the $3$-manifolds.

\smallbreak

The starting point for computing the Homflypt skein module of the lens space $L(p,1)$ is the Homflypt skein module of the solid torus, $\mathcal{S}({\rm ST})$, which is free as shown in \cite{HK} and \cite{Tu}. The reason is that ambient isotopy in ST is extended to ambient isotopy in $L(p,1)$ by adding extra moves, which reflect the surgery description of $L(p,1)$ and which are called band moves (for an illustration see Figure~\ref{bmov}). We start with a link $L$ in ST and have that $L\sim sl(L)$, where $sl(L)$ is the result of the performance of a band move on $L$.

\smallbreak

 In \cite{La3} $\mathcal{S}({\rm ST})$ is computed via knot algebras and Markov traces. More precisely, oriented links in ST are represented by the Artin braid groups of type B, $B_{1,n}$, and link isotopy in ST corresponds to braid equivalence moves in $B_{1,n}$ (\cite{La3, La4}). To the braid groups $B_{1,n}$ are associated Hecke type knot quotient algebras: the Hecke algebra of type B, $H_n(q,Q)$, the cyclotomic Hecke algebras of type B, $H_n(q,d)$, and the, so called in \cite{La3}, generalized Hecke algebras of type B, $H_{1,n}(q)$. Note that in \cite{La3} $H_{1,n}(q)$ is denoted as $H_{n}(q, \infty)$. Then, the universal analogue of the Homflypt polynomial, $X$, for links in the solid torus ST, is obtained from the generalized Hecke algebras of type B via a unique Markov trace constructed on them, which recovers $\mathcal{S}({\rm ST})$. In the algebraic language of \cite{La3} the basis of $\mathcal{S}({\rm ST})$ is given in open braid form by the set $\Lambda^{\prime}$ in Eq.~(\ref{Lpr}) (for an illustration see Figure~\ref{basel}). The algebraic setting of $L(p,1)$ is the same as ST, and thus, the invariant $X$ is appropriate for being extended to the lens space $L(p,1)$. In order to extend $X$ to an invariant of knots and links in $L(p,1)$, we need to solve an infinite system of equations resulting from the braid band moves, that is, equivalence moves between mixed braids which are band moves between their closures. Namely, we force

\begin{equation}\label{eq}
X_{\widehat{a}}\ =\ X_{\widehat{sl(a)}}, 
\end{equation}

\noindent where $\widehat{a}$ is an element of a braid group of type B and $\widehat{sl(a)}$ denotes a braided sliding of the closed braid $\widehat{a}$.

\smallbreak

In \cite{DL1} a new basis, $\Lambda$,  of $\mathcal{S}({\rm ST})$ is presented (see Theorem~\ref{newbasis} in this paper). For an illustration see Figure~\ref{basel}. The formulations of Eq.~(\ref{eq}) become particularly simple if one considers elements in the basis $\Lambda$, since braid band moves can be naturally described by elements in that basic set (see Figure~\ref{bbm12}).

\smallbreak

In this paper we present all topological steps needed in order to relate $\mathcal{S}(L(p,1))$ to $\mathcal{S}({\rm ST})$. We also present an augmented set $L\supset \Lambda$ which plays a crucial role on this paper.

\smallbreak

Our strategy is based on the following steps:

\begin{itemize}
\item[$\bullet$] By linearity, Equations~\ref{eq} boil down to considering only words in the canonical basis of the algebra $H_{1,n}(q)$, $\Sigma^{\prime}_{n}$. 
\item[$\bullet$] For words in $\Sigma^{\prime}_{n}$ we have to solve the equations $X_{\widehat{\alpha^{\prime}}}=X_{\widehat{sl_{\pm 1}(\alpha^{\prime})}}$, where $\widehat{sl_{\pm 1}(\alpha^{\prime})}$ is the result of the performance of a braid band move on the first moving strand of the closed braid $\widehat{\alpha^{\prime}}$, and $\alpha^{\prime} \in \Sigma^{\prime}_{n}$.
\item[$\bullet$] We then express elements in $\Sigma^{\prime}_{n}$ to elements in the linear bases of ${\rm H}_{1,n}(q)$, $\Sigma_n$, and show that the equations described in step 2 are equivalent to equations of the form $X_{\widehat{\alpha}}=X_{\widehat{sl_{\pm 1}(\alpha)}}$, where $\alpha \in \Sigma_n$.
\item[$\bullet$] Starting now from elements in $\Sigma_{n}$ we reduce the equations described in step 3 to equations obtained from elements in the ${\rm H}_{1,n}(q)$-module $L$, where the braid band moves are performed on any moving strand. Namely, $X_{\widehat{\beta}}=X_{\widehat{sl_{\pm i}(\beta)}}$, where $\widehat{sl_{\pm i}(\beta)}$ is the result of the performance of a braid band move on the $i^{th}$ moving strand of the closed braid $\widehat{\beta}$, and $\beta$ an element in the augmented set $L$ followed by a ``braiding tail''.
\item[$\bullet$] Then, we reduce the equations obtained from elements in the ${\rm H}_{1,n}(q)$-module $L$ by performing braid band moves on any strand, to equations obtained from elements in the ${\rm H}_{1,n}(q)$-module $\Lambda$ (\cite{DL1}) by performing braid band moves on any strand.
\item[$\bullet$] We eliminate now the ``braiding tails'' from elements in the ${\rm H}_{1,n}(q)$-module $\Lambda$ and reduce the computations to the basis of $\mathcal{S}({\rm ST})$, $\Lambda$, which is better adopted to band moves. 
\item[$\bullet$] Finally, the computation of the Homflypt skein module of the lens spaces $L(p,1)$, reduces to solving the infinite system of equations obtained from elements in the basis of $\mathcal{S}({\rm ST})$, $\Lambda$, by performing braid band moves on every moving strand.
\end{itemize}

\bigbreak

In \cite{DL3} we deal with the solution of this infinite system and compute $\mathcal{S}\left(L(p,1) \right)$. The paper is organized as follows: In Section~1 we recall the algebraic setting and results needed from \cite{La2, LR2, DL1}. We present the generalized Iwahori-Hecke algebra of type B, which is related to the knot theory of the solid torus and which plays a crucial role for this paper. We discuss its properties and present the Homflypt skein module of the solid torus (via braids). In Section~2 we start from diagrams in ST and show that in order to compute $\mathcal{S}(L(p,1))$, it suffices to consider elements in the linear bases of the algebras ${\rm H}_{1,n}(q)$, $\Sigma_n$. Finally, in Section~3 we reduce the computations only on elements in the basis $\Lambda$ of $\mathcal{S}({\rm ST})$and we arrive at the infinite system of equations, the solution of which is equivalent to computing $\mathcal{S}(L(p,1))$.

\bigbreak

In \cite{GM} the Homflypt skein module of the lens spaces $L(p,1)$ is computed using diagrammatic method. The diagrammatic method could in theory be generalized to the case of $L(p,q), q > 1$, but the diagrams become even more cumbersome to analyze and several induction arguments fail. The importance of our approach is that it can shed light to the problem of computing skein modules of arbitrary c.c.o. $3$-manifolds, since any $3$-manifold can be obtained by surgery on $S^3$ along unknotted closed curves, and since braid band moves are much more controlled than the ones in the diagrammatic setting. Indeed, one can use the results presented here in order to apply a braid approach to the skein module of an arbitrary c.c.o. $3$-manifold. The advantage of the braid approach is that it gives more control over the band moves than the diagrammatic approach and much of the diagrammatic complexity is absorbed into the proofs of the algebraic statements. The main difficulty of the problem lies in selecting from the infinitum of band moves (or handle slide moves) some basic ones, solving the infinite system of equations and proving that there are no dependencies in the solutions.

\section{Topological and algebraic tools}

\subsection{Mixed links and isotopy in $L(p,1)$}

We consider ST to be the complement of a solid torus in $S^3$. An oriented link $L$ in ST can be represented by an oriented \textit{mixed link} in $S^{3}$, that is a link in $S^{3}$ consisting of the unknotted fixed part $\widehat{I}$ representing the complementary solid torus in $S^3$ and the moving part $L$ that links with $\widehat{I}$. A \textit{mixed link diagram }is a diagram $\widehat{I}\cup \widetilde{L}$ of $\widehat{I}\cup L$ on the plane of $\widehat{I}$, where this plane is equipped with the top-to-bottom direction of $I$.

\begin{figure}
\begin{center}
\includegraphics[width=1.1in]{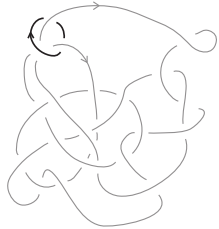}
\end{center}
\caption{A mixed link in $S^3$.}
\label{mlink}
\end{figure}

\smallbreak

It is well known that the lens spaces $L(p,1)$ can be obtained from $S^3$ by surgery on the unknot with surgery coefficient $p$. Surgery along the unknot can be realized by considering first the complementary solid torus and then attaching to it a solid torus according to some homeomorphism on the boundary. Thus, isotopy in $L(p,1)$ can be viewed as isotopy in ST together with the band moves in $S^3$, which reflects the surgery description of the manifold (see Figure~\ref{bmov}). In \cite{DL2} we show that in order to describe isotopy for knots and links in a c.c.o. $3$-manifold, it suffices to consider only the type $a$ band moves (for an illustration see Fig.~\ref{bmov}) and thus, isotopy between oriented links in $L(p,1)$ is reflected in $S^3$ by means of the following result (cf. Thm.~5.8 \cite{LR1}, Thm.~6 \cite{DL2} ):

\smallbreak

{\it
Two oriented links in $L(p,1)$ are isotopic if and only if two corresponding mixed link diagrams of theirs differ by isotopy in {\rm ST} together with a finite sequence of the type $a$ band moves.
}

\begin{figure}
\begin{center}
\includegraphics[width=4.5in]{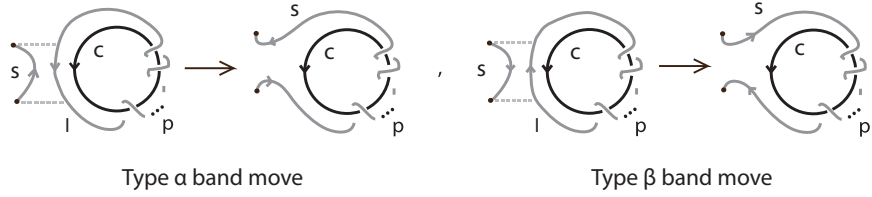}
\end{center}
\caption{The two types of band moves.}
\label{bmov}
\end{figure}

\subsection{Mixed braids and braid equivalence for knots and links in $L(p,1)$}

By the Alexander theorem for knots in solid torus (cf. Thm.~1 \cite{La4}), a mixed link diagram $\widehat{I}\cup \widetilde{L}$ of $\widehat{I}\cup L$ may be turned into a \textit{mixed braid} $I\cup \beta $ with isotopic closure. This is a braid in $S^{3}$ where, without loss of generality, its first strand represents $\widehat{I}$, the fixed part, and the other strands, $\beta$, represent the moving part $L$. The subbraid $\beta$ is called the \textit{moving part} of $I\cup \beta $ (see Fig.~\ref{mbtoml}).

\begin{figure}
\begin{center}
\includegraphics[width=2in]{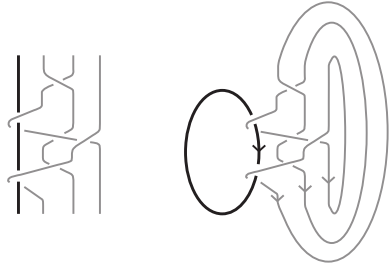}
\end{center}
\caption{The closure of a mixed braid to a mixed link.}
\label{mbtoml}
\end{figure}

Then, in order to translate isotopy for links in $L(p,1)$ into braid equivalence, we first perform the technique of {\it standard parting} introduced in \cite{LR2}: Parting a geometric mixed braid means to separate the moving strands from the fixed strand that represents the lens spaces $L(p,1)$. This can be realized by pulling each pair of corresponding moving strands to the right and {\it over\/} or {\it under\/} the fixed strand that lies on their right. Then, we define a {\it braid band move} to be a move between mixed braids, which is a band move between their closures. It starts with a little band oriented downward, which, before sliding along a surgery strand, gets one twist {\it positive\/} or {\it negative\/} (see Figure~\ref{bbmfig}).

\begin{figure}
\begin{center}
\includegraphics[width=1.9in]{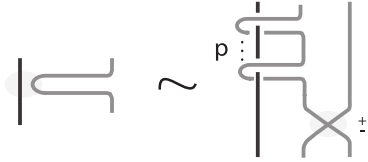}
\end{center}
\caption{The two types of braid band moves.}
\label{bbmfig}
\end{figure}

\smallbreak

The sets of braids related to the ST form groups, which are in fact the Artin braid groups type B, denoted $B_{1,n}$, with presentation:

\[ B_{1,n} = \left< \begin{array}{ll}  \begin{array}{l} t, \sigma_{1}, \ldots ,\sigma_{n-1}  \\ \end{array} & \left| \begin{array}{l}
\sigma_{1}t\sigma_{1}t=t\sigma_{1}t\sigma_{1} \ \   \\
 t\sigma_{i}=\sigma_{i}t, \quad{i>1}  \\
{\sigma_i}\sigma_{i+1}{\sigma_i}=\sigma_{i+1}{\sigma_i}\sigma_{i+1}, \quad{ 1 \leq i \leq n-2}   \\
 {\sigma_i}{\sigma_j}={\sigma_j}{\sigma_i}, \quad{|i-j|>1}  \\
\end{array} \right.  \end{array} \right>, \]

\noindent where the generators $\sigma _{i}$ and $t$ are illustrated in Figure~\ref{gen}.

\begin{figure}
\begin{center}
\includegraphics[width=2.3in]{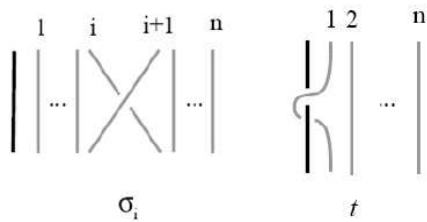}
\end{center}
\caption{The generators of $B_{1,n}$.}
\label{gen}
\end{figure}

In \cite{LR1} the authors give a sharpened version of the classical Markov's theorem based only on one type of moves, the $L$-moves: An {\it $L$-move} on a mixed braid $B \bigcup \beta$, consists in cutting an arc of the moving subbraid open and pulling the upper cutpoint downward and
the lower upward, so as to create a new pair of braid strands with corresponding endpoints, and such that both strands cross entirely {\it over} or {\it under} with the rest of the braid. Stretching the new strands over will give rise to an {\it $L_o$-move\/} and under to an {\it $L_u$-move\/}. For an illustration see Figure~\ref{lmove}.

\begin{figure}
\begin{center}
\includegraphics[width=4.4in]{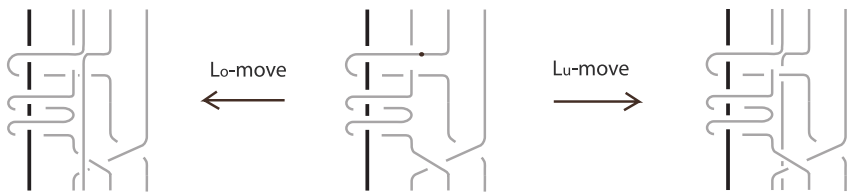}
\end{center}
\caption{A mixed braid and the two types of $L$-moves}
\label{lmove}
\end{figure}

\bigbreak

Isotopy in $L(p,1)$ then is translated on the level of mixed braids by means of the following theorem:

\begin{thm}[Theorem~5, \cite{LR2}] \label{markov}
 Let $L_{1} ,L_{2}$ be two oriented links in $L(p,1)$ and let $I\cup \beta_{1} ,{\rm \; }I\cup \beta_{2}$ be two corresponding mixed braids in $S^{3}$. Then $L_{1}$ is isotopic to $L_{2}$ in {\rm ST} if and only if $I\cup \beta_{1}$ is equivalent to $I\cup \beta_{2}$ in $\mathop{\cup }\limits_{n=1}^{\infty } B_{1,n}$ by the following moves:
\[ \begin{array}{clll}
(i)  & Conjugation:         & \alpha \sim \beta^{-1} \alpha \beta, & {\rm if}\ \alpha ,\beta \in B_{1,n}. \\
(ii) & Stabilization\ moves: &  \alpha \sim \alpha \sigma_{n}^{\pm 1} \in B_{1,n+1}, & {\rm if}\ \alpha \in B_{1,n}. \\
(iii) & Loop\ conjugation: & \alpha \sim t^{\pm 1} \alpha t^{\mp 1}, & {\rm if}\ \alpha \in B_{1,n}. \\
(iv) & Braid\ band\ moves: & \alpha \sim {t^{\prime}_n}^p \sigma_n^{\pm 1}\alpha^{\prime}, & a^{\prime}\in B_{1, n+1},
\end{array} \]

\noindent {\rm where} $a^{\prime}$ {\rm is the word }  $a$  {\rm with the substitutions:}

$$t^{\pm 1} \longleftrightarrow (\sigma_1^{-1} \ldots \sigma_{n-1}^{-1} \cdot  \sigma_n^2\cdot \sigma_{n-1}\ldots \sigma_{1}t)^{\pm 1} $$
\noindent Equivalently, by the same moves as above, where (i) and (ii) are replaced by the two types of $L$-moves.
\end{thm}

In the statement of Theorem~\ref{markov} the braid band moves take place on the last strand of a mixed braid. Clearly, this is equivalent to performing the braid band moves on the first moving strand (see Figure~\ref{bbmconj}) or, in fact, on any specified moving strand of the mixed braid. Indeed we have the following:

\begin{lemma}\label{bbm1}
A braid band move may always be assumed to be performed on the first moving strand of a mixed braid. Moreover, such a braid band move 
can be expressed algebraically by the following relation 
 
$$\alpha \sim t^p \sigma_1^{\pm 1}\alpha_+,$$

\noindent where $\alpha_+$ is $\alpha$ with all indices shifted by $1$.
\end{lemma}

\begin{proof}
In a mixed braid $I\cup \beta$ consider the last strand of $\beta$ approaching the surgery strand $I$ from the right. Before performing a bbm we apply conjugation (isotopy in ST) and obtain an equivalent mixed braid where the first strand is now approaching $I$ (see Figure~\ref{bbmconj}). In terms of diagrams we have the following:

\[
\begin{array}{ccccccc}
\beta & \sim & (\sigma_{i-1}\ldots \sigma_1 \sigma_1^{-1}\ldots \sigma_{i-1}^{-1})\cdot \beta {\sim} \ \underset{\alpha}{\underbrace{(\sigma_1^{-1}\ldots \sigma_{i-1}^{-1})\cdot \beta \cdot (\sigma_{i-1}\ldots \sigma_1)}} & = &  \alpha\\
\downarrow &  &  &   & \downarrow\\
sl_{\pm i}(\beta) &  & = &  & sl_{\pm 1}(\alpha)
\end{array}
\]

The proof of the second statement of the lemma is clear by viewing Figure~\ref{bbm12}.

\begin{figure}
\begin{center}
\includegraphics[width=4in]{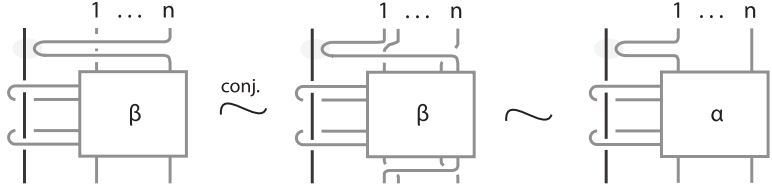}
\end{center}
\caption{ Proof of Lemma~\ref{bbm1}. }
\label{bbmconj}
\end{figure}

\begin{figure}
\begin{center}
\includegraphics[width=2.5in]{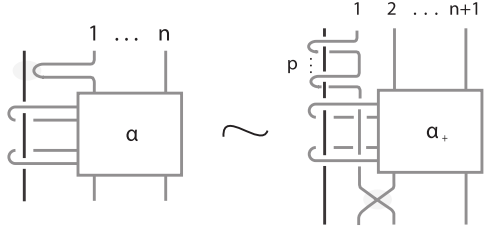}
\end{center}
\caption{Braid band move performed on the first moving strand.}
\label{bbm12}
\end{figure}

\end{proof}

\begin{nt}
We denote a braid band move by bbm and, specifically, the result of a positive or negative braid band move performed on the $i^{th}$-moving strand of a mixed braid $\beta$ by $sl_{\pm i}(\beta)$.
\end{nt}

\subsection{The generalized Iwahori-Hecke algebra of type B}

It is well known that $B_{1,n}$ is the Artin group of the Coxeter group of type B, which is related to the Hecke algebra of type B, $\textrm{H}_{n}{(q,Q)}$ and to the cyclotomic Hecke algebras of type B. In \cite{La3} it has been established that all these algebras form a tower of B-type algebras and are related to the knot theory of ST. The basic one is $\textrm{H}_{n}{(q,Q)}$, a presentation of which is obtained from the presentation of the Artin group $B_{1,n}$ by adding the quadratic relations
\begin{equation}\label{quad}
{g_{i}^2=(q-1)g_{i}+q}
\end{equation}

\noindent and the relation $t^{2} =\left(Q-1\right)t+Q$, where $q,Q \in {\mathbb C}\backslash \{0\}$ are seen as fixed variables. The middle B--type algebras are the cyclotomic Hecke algebras of type B, $\textrm{H}_{n}(q,d)$, whose presentations are obtained by the quadratic relation~(\ref{quad}) and $t^d=(t-u_{1})(t-u_{2}) \ldots (t-u_{d})$. The topmost Hecke-like algebra in the tower is the \textit{generalized Iwahori--Hecke algebra of type B}, $\textrm{H}_{1,n}(q)$, which, as observed by T. tom Dieck, is closely related to the affine Hecke algebra of type A, $\widetilde{\textrm{H}}_n(q)$ (cf. \cite{La3}). The algebra $\textrm{H}_{1,n}(q)$ has the following presentation:
\[
\textrm{H}_{1,n}{(q)} = \left< \begin{array}{ll}  \begin{array}{l} t, g_{1}, \ldots ,g_{n-1}  \\ \end{array} & \left| \begin{array}{l} g_{1}tg_{1}t=tg_{1}tg_{1} \ \
\\
 tg_{i}=g_{i}t, \quad{i>1}  \\
{g_i}g_{i+1}{g_i}=g_{i+1}{g_i}g_{i+1}, \quad{1 \leq i \leq n-2}   \\
 {g_i}{g_j}={g_j}{g_i}, \quad{|i-j|>1}  \\
 {g_i}^2=(q-1)g_{i}+q, \quad{i=1,\ldots,n-1}
\end{array} \right.  \end{array} \right>.
\]
\noindent That is:
\begin{equation*}
\textrm{H}_{1,n}(q)= \frac{{\mathbb Z}\left[q^{\pm 1} \right]B_{1,n}}{ \langle \sigma_i^2 -\left(q-1\right)\sigma_i-q \rangle}.
\end{equation*}

Note that in $\textrm{H}_{1,n}(q)$ the generator $t$ satisfies no polynomial relation, making the algebra $\textrm{H}_{1,n}(q)$ infinite dimensional. Also that in \cite{La3} the algebra $\textrm{H}_{1,n}(q)$ is denoted as $\textrm{H}_{n}(q, \infty)$.

\smallbreak

In \cite{Jo} V.F.R. Jones gives the following linear basis for the Iwahori-Hecke algebra of type A, $\textrm{H}_{n}(q)$:

$$ S =\left\{(g_{i_{1} }g_{i_{1}-1}\ldots g_{i_{1}-k_{1}})(g_{i_{2} }g_{i_{2}-1 }\ldots g_{i_{2}-k_{2}})\ldots (g_{i_{p} }g_{i_{p}-1 }\ldots g_{i_{p}-k_{p}})\right\}, $$

\noindent for $1\le i_{1}<\ldots <i_{p} \le n-1{\rm \; }$.

\noindent The basis $S$ yields directly an inductive basis for $\textrm{H}_{n}(q)$, which is used in the construction of the Ocneanu trace, leading to the Homflypt or $2$-variable Jones polynomial.

In $\textrm{H}_{1,n}(q)$ we define the elements:
\begin{equation}\label{lgen}
t_{i}:=g_{i}g_{i-1}\ldots g_{1}tg_{1} \ldots g_{i-1}g_{i}\ \rm{and}\ t^{\prime}_{i}:=g_{i}g_{i-1}\ldots g_{1}tg_{1}^{-1}\ldots g_{i-1}^{-1}g_{i}^{-1},
\end{equation}
as illustrated in Figure~\ref{genh}.

\smallbreak

In \cite{La3} the following result has been proved.

\begin{thm}[Proposition~1, Theorem~1 \cite{La3}] \label{basesH}
The following sets form linear bases for ${\rm H}_{1,n}(q)$:
\[
\begin{array}{llll}
 (i) & \Sigma_{n} & = & \{t_{i_{1} } ^{k_{1} } t_{i_{2} } ^{k_{2} } \ldots t_{i_{r}}^{k_{r} } \cdot \sigma \} ,\ {\rm where}\ 1\le i_{1} <\ldots <i_{r} \le n-1,\\
 (ii) & \Sigma^{\prime} _{n} & = & \{ {t^{\prime}_{i_1}}^{k_{1}} {t^{\prime}_{i_2}}^{k_{2}} \ldots {t^{\prime}_{i_r}}^{k_{r}} \cdot \sigma \} ,\ {\rm where}\ 1\le i_{1} < \ldots <i_{r} \le n, \\
\end{array}
\]
\noindent where $k_{1}, \ldots ,k_{r} \in {\mathbb Z}$ and $\sigma$ a basic element in $\textrm{H}_{n}(q)$.
\end{thm}

\begin{figure}
\begin{center}
\includegraphics[width=3.2in]{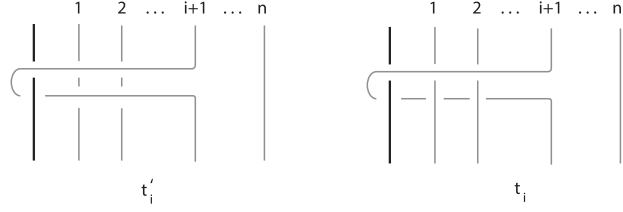}
\end{center}
\caption{The elements $t^{\prime}_{i}$ and $t_{i}$.}
\label{genh}
\end{figure}

\begin{remark}\label{conind}\rm
The indices of the $t^{\prime}_i$'s in the set $\Sigma^{\prime}_n$ are ordered but are not necessarily consecutive, neither do they
need to start from $t$.
\end{remark}

\subsection{The Homflypt skein module of ST}

In \cite{La3} the basis $\Sigma^{\prime}_{n}$ is used for constructing a Markov trace on $\bigcup _{n=1}^{\infty }\textrm{H}_{1,n}(q)$.

\begin{thm}[Theorem~6, \cite{La3}] \label{tr}
Given $z,s_{k}$, with $k\in {\mathbb Z}$ specified elements in $R={\mathbb Z}\left[q^{\pm 1} \right]$, there exists
a unique linear Markov trace function
\begin{equation*}
{\rm tr}:\bigcup _{n=1}^{\infty }{\rm H}_{1,n}(q)  \to R\left(z,s_{k} \right),k\in {\mathbb Z}
\end{equation*}

\noindent determined by the rules:
\[
\begin{array}{lllll}
(1) & {\rm tr}(ab) & = & {\rm tr}(ba) & \quad {\rm for}\ a,b \in {\rm H}_{1,n}(q) \\
(2) & {\rm tr}(1) & = & 1 & \quad {\rm for\ all}\ {\rm H}_{1,n}(q) \\
(3) & {\rm tr}(ag_{n}) & = & z{\rm tr}(a) & \quad {\rm for}\ a \in {\rm H}_{1,n}(q) \\
(4) & {\rm tr}(a{t^{\prime}_{n}}^{k}) & = & s_{k}{\rm tr}(a) & \quad {\rm for}\ a \in {\rm H}_{1,n}(q),\ k \in {\mathbb Z}. \\
\end{array}
\]
\end{thm}

Note that the use of the looping elements $t_i^{\prime}$ enable the trace ${\rm tr}$ to be defined by just extending the three rules of the Ocneanu trace on the algebras ${\rm H}_n(q)$ \cite{Jo} by rule (4). Using $\textrm{tr}$ Lambropoulou constructed a universal Homflypt-type invariant for oriented links in ST. Namely, let $\mathcal{L}$ denote the set of oriented links in ST. Then:

\begin{thm} [Definition~1, \cite{La3}] \label{inv}
The function $X:\mathcal{L}$ $\rightarrow R(z,s_{k})$

\begin{equation*}
X_{\widehat{\alpha}}=\left[-\frac{1-\lambda q}{\sqrt{\lambda } \left(1-q\right)} \right]^{n-1} \left(\sqrt{\lambda } \right)^{e}
{\rm tr}\left(\pi \left(\alpha \right)\right),
\end{equation*}

\noindent where $\lambda := \frac{z+1-q}{qz}$, $\alpha \in B_{1,n}$ is a word in the $\sigma _{i}$'s and $t^{\prime}_{i} $'s, $\widehat{\alpha}$ is the closure of $\alpha$, $e$ is the exponent sum of the $\sigma _{i}$'s in $\alpha $, and $\pi$ the canonical map of $B_{1,n}$ in ${\rm H}_{1,n}(q)$, such that $t\mapsto t$ and $\sigma _{i} \mapsto g_{i} $, is an invariant of oriented links in {\rm ST}.
\end{thm}

\bigbreak

In \cite{Tu, HK} ST was considered as ${\rm (Annulus)} \times {\rm (Interval)}$. In our braid setting, the elements of $\mathcal{S}({\rm ST})$ correspond bijectively to the elements of the following set $\Lambda^{\prime}$:

\begin{equation}\label{Lpr}
\Lambda^{\prime}=\{ {t^{k_0}}{t^{\prime}_1}^{k_1}{t^{\prime}_2}^{k_2} \ldots
{t^{\prime}_n}^{k_n}, \ k_i \in \mathbb{Z}\setminus\{0\}, \ k_i \geq k_{i+1}\ \forall i,\ n\in \mathbb{N} \}.
\end{equation}

\noindent So, we have that $\Lambda^{\prime}$ is a basis of $\mathcal{S}({\rm ST})$ in terms of braids. Note that $\Lambda^{\prime}$ is a subset of $\bigcup_n{\textrm{H}_{1,n}}$ and, in particular, $\Lambda^{\prime}$ is a subset of $\bigcup_n{\Sigma^{\prime}_n}$. Note also that in contrast to elements in $\bigcup_n{\Sigma^{\prime}_n}$, the elements in $\Lambda^{\prime}$ have no gaps in the indices, the exponents are ordered and there are no ``braiding tails''. The Lambropoulou invariant $X$ recovers $\mathcal{S}({\rm ST})$, because it gives distinct values to distinct elements, since $tr(t^{k_0}{t^{\prime}_1}^{k_1}{t^{\prime}_2}^{k_2} \ldots {t^{\prime}_n}^{k_n})=s_{k_n}s_{k_{n-1}}\ldots s_{k_1}s_{k_0}$.

\subsection{The new basis, $\Lambda$, of $\mathcal{S}({\rm ST})$}\label{lamb}

In \cite{DL1} we give a different basis $\Lambda$ for $\mathcal{S}({\rm ST})$, which was predicted by the third author. The new basic set is described in Eq.~\ref{basis} in open braid form. The looping elements $t_i$ are in the algebras $\textrm{H}_{1,n}(q)$ and they commute. Moreover, the $t_i$'s are consistent with the braid band move used in the link isotopy in $L(p,1)$, in the sense that a bbm can be described naturally with the use of the $t_i$'s (see for example \cite{DL1} and references therein).

\begin{thm}[DL, Theorem~2]\label{newbasis}
The following set is a $\mathbb{Z}[q^{\pm1}, z^{\pm1}]$-basis for $\mathcal{S}({\rm ST})$:
\begin{equation}\label{basis}
\Lambda=\{t^{k_0}t_1^{k_1}\ldots t_n^{k_n},\ k_i \in \mathbb{Z}\setminus\{0\},\ k_i \geq k_{i+1}\ \forall i,\ n \in \mathbb{N} \}.
\end{equation}
\end{thm}

\begin{figure}
\begin{center}
\includegraphics[width=1.6in]{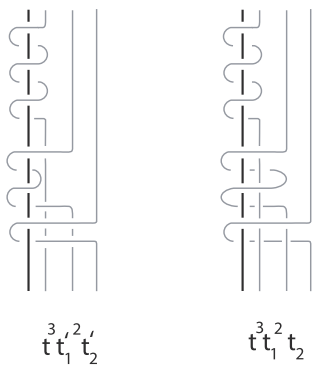}
\end{center}
\caption{Elements in two different bases of $\mathcal{S}({\rm ST})$.}
\label{basel}
\end{figure}

Notice that comparing the set $\Lambda$ with the sets $\Sigma_n$, we observe that there are no gaps in the indices of the $t_i$'s and the exponents are in decreasing order. Also, there are no ``braiding tails'' in the words in $\Lambda$.

\smallbreak

The importance of the new basis, $\Lambda$, of $\mathcal{S}({\rm ST})$ lies in the simplicity of the algebraic expression of a bbm, as given in Lemma~\ref{bbm1} and as illustrated in Figure~\ref{bbm12}.

\section{Reducing computations to the linear bases of the algebras ${\rm H}_{1,n}(q)$}

In order to compute $\mathcal{S}(L(p,1))$ we need to normalize the invariant $X$ by forcing it to satisfy all possible braid band moves. In the following sections we show that it suffices to consider elements in an augmented set $L$ and perform braid band moves only on the first moving strand of the mixed braids. More precisely, we first show that it suffices to consider elements in the canonical basis of the algebra $H_{1,n}(q)$, $\Sigma^{\prime}_{n}$, and perform bbm's on their first moving strand. We then pass from the linear basis of of ${\rm H}_{1,n}(q)$, $\Sigma_n$ and reduce the equations obtained from elements in $\Sigma_n^{\prime}$ to equations obtained from elements in $\Sigma$ by considering the performance of bbm's on their first moving strand. In order to reduce the computation to elements in the basis of $\mathcal{S}({\rm ST})$, $\Lambda$, we first order the exponents of the looping generators of elements in $\Sigma_n$ and obtain elements in the augmented set $L$, followed by ``braiding tails''. Note that the performance of braid band moves is now considered to take place on any moving strand. Then, we show that the equations obtained from elements in the ${\rm H}_{1,n}(q)$-module $L$ by performing braid band moves on any strand, are equivalent to equations obtained from elements in the ${\rm H}_{1,n}(q)$-module $\Lambda$ by performing braid band moves on any strand. We eliminate now the ``braiding tails'' from elements in the ${\rm H}_{1,n}(q)$-module $\Lambda$ and reduce the computations to the set $\Lambda$, where bbm's are performed on any moving strand.

\subsection{Reducing to braid band moves on $\Sigma^{\prime}$}

From now on we shall denote by $\Sigma$ the set $\cup_n \Sigma_n$ and similarly, by $\Sigma^{\prime}$ we shall denote the set $\cup_n \Sigma_n^{\prime}$. We now show that it suffices to perform bbm's on elements in the linear basis of ${\rm H}_{1,n}(q)$, $\Sigma^{\prime}$.
This is the first step in order to restrict the performance of bbm's only on elements in $\Lambda$. For this we need the following lemma:

\begin{lemma}\label{lbbm'skein}
Braid band moves and skein relation are interchangeable, that is, for $\tau_1^{\prime}\in \Sigma^{\prime}$ and $w \in {\rm H}_{n}(q)$ the following diagram commutes:
$$
\begin{CD}
\tau_1^{\prime} \cdot w @>(\pm)(p,1) bbm>> \tau_2^{\prime}\cdot w_+ g_1^{\pm1}\\
@VVquadraticV @VVquadraticV\\
\sum_j f_j(q)\tau_1^{\prime}\cdot w_j @>(\pm)(p,1) bbm>> \sum_j f_j(q) \tau_2^{\prime}\cdot w_{j_+} g_1^{\pm1}
\end{CD}
$$
\end{lemma}

\begin{proof}
Let $\tau_1^{\prime}$ a monomial in $t_i^{\prime}$'s, $w\in {\rm H}_{n}(q)$ such that $w=\sum_{i=1}^{n}{f_{i}(q)w_{i}}$, where $w_i$ are words in canonical form and $f_{i}(q)$ a one parameter expressions in $\mathbb{C}$ for all $i$. We perform a braid band move on $\tau_1^{\prime}\cdot w$ and obtain:

\[
\tau_1^{\prime}\cdot w \ \xrightarrow[bbm]{(\pm)(p,1)} \ \tau_2^{\prime}\cdot w_+g_1^{\pm1},
\]

\noindent where $w_+=\sum_{i=1}^{n}{f_{i}(q)w_{i_+}}$. Then:

\[
\tau_1^{\prime}\cdot w \ = \ \tau_1^{\prime}\cdot \sum_{i=1}^{n}{f_{i}(q)w_{i}}\ \xrightarrow[bbm]{(\pm)(p,1)} \ \tau_2^{\prime}\cdot \sum_{i=1}^{n}{f_{i}(q)w_{i_+}} g_1^{\pm1}.
\]

\noindent We also have that:

\[
\tau_1^{\prime}\cdot w_j \ \xrightarrow[bbm]{(\pm)(p,1)}\ \tau_2^{\prime}\cdot w_{j_+} g_1^{\pm1}\ \forall\ j,\ \text{and thus}
\]

\[
\tau_2^{\prime}\cdot \sum_{i=1}^{n}{f_{i}(q)w_{i}} g_1^{\pm1}\ =\ \tau_2^{\prime}\cdot w_{+} g_1^{\pm1},
\]
\noindent and this concludes the proof (see also Figure~\ref{bbm'skein}).
\end{proof}

\begin{figure}
\begin{center}
\includegraphics[width=4.5in]{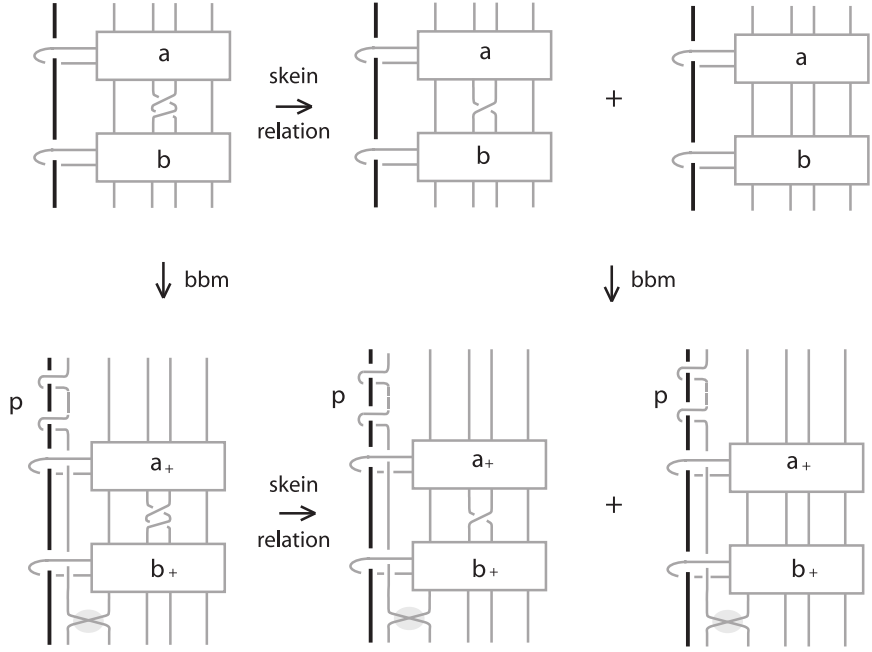}
\end{center}
\caption{Proof of Lemma~\ref{lbbm'skein}.}
\label{bbm'skein}
\end{figure}

Using Lemma~\ref{lbbm'skein}, we have the following:

\begin{prop}\label{bbm's'}
It suffices to consider the performance of braid band moves on the first strand of only elements in the sets $\Sigma^{\prime}$.
\end{prop}

\begin{proof}
By Artin's combing we can write words in $B_{1,n}$ in the form $\tau^{\prime}\cdot w$, where $\tau^{\prime}$ is a monomial in $t_i^{\prime}$'s and $w \in B_n$.
By Lemma~\ref{lbbm'skein} we have that:
\[
\begin{array}{rcll}
X_{\widehat{\tau^{\prime}\cdot w}} & = & X_{\widehat{t^p \tau^{\prime \prime}\cdot \sigma_1^{\pm 1} \cdot w_{+}}} &\overset{(L.~\ref{lbbm'skein})}{\Rightarrow} \\
&&&\\
\sum_i{A_i \cdot X_{\widehat{\tau^{\prime}\cdot w_i}}} & = & \sum_i{A_i \cdot X_{\widehat{t^p \tau^{\prime \prime}\cdot \sigma_1^{\pm 1} \cdot w_{i_{+}}}}}, & \\
\end{array}
\]

\noindent where $w_i$ are words in reduced form in ${\rm H}_{1,n}(q),\ \forall i$ and $A_i \in \mathbb{C}$.
\end{proof}

\begin{remark}
Note that even though the trace is linear, the invariant $X$ is not a linear function. Namely, if $\tau = \tau_1 + \tau_2$, where $\tau, \tau_1, \tau_2 \in \Sigma^{\prime}_n$, such that $X_{\widehat{\tau}} = A\cdot {\rm tr}(\tau), X_{\widehat{\tau_1}}=A_1\cdot {\rm tr}(\tau_1)$ and $X_{\widehat{\tau_2}}=A_2\cdot {\rm tr}(\tau_2)$, where $A, A_1, A_2\in \mathbb{C}$, then:
\[
X_{\widehat{\tau}} = A\cdot {\rm tr}(\tau) = A\cdot {\rm tr}(\tau_1 + \tau_2) = A\cdot {\rm tr}(\tau_1) + A\cdot {\rm tr}(\tau_2) \neq A_1\cdot {\rm tr}(\tau_1) + A_2\cdot {\rm tr}(\tau_2) = X_{\widehat{\tau_1}} +X_{\widehat{\tau_2}}.
\]

\end{remark}

\subsection{From the set $\Sigma^{\prime}$ to the set $\Sigma$}

We shall now show that it suffices to perform bbm's on elements in the linear bases $\Sigma_n$ of the algebras ${\rm H}_{1,n}(q)$. Their union is the set $\Sigma$ which includes as a proper subset the basis, $\Lambda$, of $\mathcal{S}({\rm ST})$, described in \S~\ref{lamb}.

\smallbreak

Let $\tau^{\prime}\cdot w \in \Sigma^{\prime}$. We have that:
\[
\begin{matrix}
\tau^{\prime}\cdot w & = & (t^{k_0} {t^{\prime}_1}^{k_1 } \ldots {t^{\prime}_m}^{k_m})\cdot w & {=} &\underset{\tau}{\underbrace{ t^{k_0} (t_1g_1^{-2})^{k_1} \ldots ({t_mg_m^{-1}\ldots g_2^{-1} g_1^{-2} g_2^{-1} \ldots g_m^{-1}})^{k_m}}}\cdot w\ =\\
  &  &  &  &   \\
  & = & \tau \cdot w  &  &    \\
\end{matrix}
\]

Perform a bbm on the first moving strand of both $\tau^{\prime}\cdot w$ and
$\tau \cdot w$ and cable the new parallel strand together with the surgery strand. Denote the result as $cbl(ps)$. Then:
\[
\begin{matrix}
\tau^{\prime}\cdot w & \overset{bbm}{\rightarrow} & cbl(ps) \cdot \tau^{\prime}\cdot w \cdot \sigma_1^{\pm 1}\\
\parallel & & \parallel \\
\tau\cdot w & \overset{bbm}{\rightarrow} & cbl(ps) \cdot \tau\cdot w \cdot \sigma_1^{\pm 1}\\
\end{matrix}
\]
So: $X_{\widehat{\tau^{\prime}\cdot w}}\ =\ X_{\widehat{sl(\tau^{\prime}\cdot w)}}\ \Leftrightarrow \ X_{\widehat{\tau\cdot w}}\ =\ X_{\widehat{sl(\tau\cdot w)}}$.
But since $\tau\cdot w \in {\rm H}_{1,n}(q)$, we can express $\tau\cdot w$ as a sum of elements in the linear basis of ${\rm H}_{1,n}(q)$, $\Sigma$, that is, $\tau\cdot w = \sum_{i}{a_i T_i \cdot w_i}$, where $T_i\cdot w_i \in \Sigma,\ \forall i$, $T_i$ a monomial in $t_i$'s with possible gaps in the indices and
unordered exponents, and $a_i \in \mathbb{C},\ \forall i$.
\[
\begin{array}{lllll}
X_{\widehat{\tau\cdot w}}\ = \ X_{\widehat{sl(\tau\cdot w)}} & \Rightarrow & \alpha \cdot tr(\tau\cdot w) & = & b\cdot tr\left(cbl(ps) \tau \cdot w \cdot g_1^{\pm1} \right)\\
& \Rightarrow & \alpha \cdot \sum_{i}{a_i\cdot tr(T_i\cdot w_i)} & = & b \cdot \sum_{i}{tr\left(a_i \cdot cbl(ps)T_i\cdot w_i\cdot g_1^{\pm1} \right)} \\
\end{array}
\]

We conclude that:

\[
\begin{matrix}
\tau^{\prime}\cdot w & \overset{bbm}{\rightarrow} & cbl(ps) \cdot \tau^{\prime}\cdot w \cdot \sigma_1^{\pm 1} & (\ast) \\
\parallel & & \parallel &\\
\tau\cdot w & \overset{bbm}{\rightarrow} & cbl(ps) \cdot \tau\cdot w \cdot \sigma_1^{\pm 1}&\\
\parallel & & \parallel &\\
\sum_{i}{a_i\cdot T_i\cdot w_i} & \overset{bbm}{\rightarrow} & \sum_{i}{a_i\cdot t^p{T_i}_{+}\cdot {w_i}_{+}g_1^{\pm1}}& (\ast \ast)\\
\end{matrix}
\]

The above are summarized in the following proposition (see also Figure~\ref{prop2a}):

\begin{prop}\label{picprop2}
The equations

\begin{equation}\label{eq1}
X_{\widehat{T^{\prime}\cdot w}} \ =\ X_{\widehat{t^p T^{\prime \prime}\cdot g_1^{\pm 1}\cdot w_+}}
\end{equation}

\noindent result from equations of the form 

\begin{equation}\label{eq2}
X_{\widehat{T\cdot w}} \ =\ X_{\widehat{t^p T_+ \cdot g_1^{\pm 1}\cdot w_+}},
\end{equation}

\noindent where $T\cdot w \in \Sigma, \ \forall i$.
\end{prop}

\begin{figure}
\begin{center}
\includegraphics[width=2.2in]{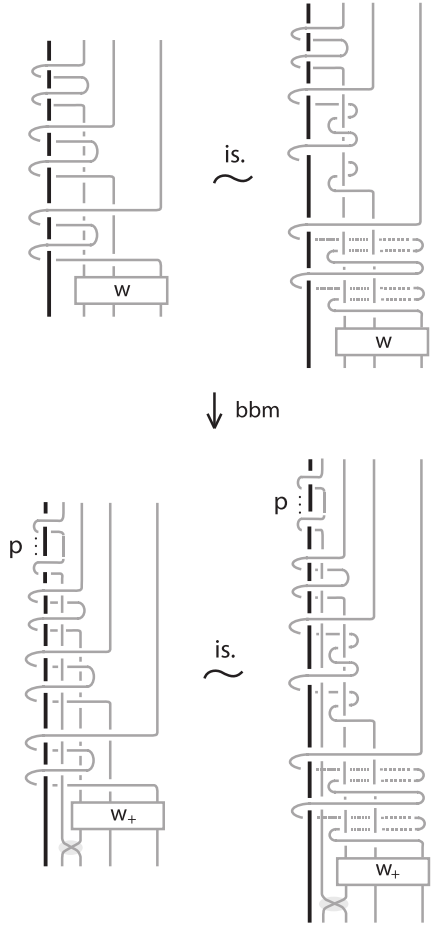}
\end{center}
\caption{Proof of Proposition~\ref{picprop2}. }
\label{prop2a}
\end{figure}

Note that elements in $\Sigma$ consist of two parts:
\begin{itemize}
\item[$\bullet$] A monomial in $t_i$'s with possible gaps in the indices and unordered exponents,
\end{itemize}
\smallbreak
\noindent followed by
\smallbreak
\begin{itemize}
\item[$\bullet$] a ``braiding tail'' in the basis of ${\rm H}_n(q)$.
\end{itemize}

In order to prove that the system obtained from elements in $\Sigma$ is equivalent to the system obtained from elements in $\Lambda$, we first manage
the gaps in the indices in the monomials in $t_i$'s, we then order the exponents and finally we eliminate the tails. The procedure is similar to the one
described in \cite{DL1}, but in this case we do that simultaneously before and after the performance of a braid band move and show that the equations obtained
from elements in the sets  $\Sigma$ and $\Lambda$ are equivalent.

\section{Reducing computations to the basis $\Lambda$ of $\mathcal{S}({\rm ST})$}

%\subsection{From $\Sigma$ to the ${\rm H}_{n}(q)$-module $\Lambda$: managing the gaps and ordering the exponents}

In order to reach to equations obtained from elements in the set $\Lambda$, we follow the same steps as in the proof of Theorem~\ref{newbasis} \cite{DL1}. The proof of Theorem~\ref{newbasis} is based on the following idea: We start with an element $\lambda^{\prime}\in \Lambda^{\prime}$ and convert it into a linear combination of elements in $\Lambda$. First we pass by elements in the sets $\Sigma$. This means that in the converted expression of $\lambda^{\prime}$ we have monomials in the $t_i$'s, with possible gaps in the indices and possible non ordered exponents followed by monomials in the braiding generators $g_i$. So, in order to reach expressions in the set $\Lambda$ we first manage the gaps in the indices on the $t_i$'s. For this we define an ordering relation in the sets $\Sigma^{\prime}$ and $\Sigma$, which include $\Lambda^{\prime}$ and $\Lambda$ as subsets.

\subsection{An ordering for the sets $\Sigma^{\prime}, \Sigma, \Lambda^{\prime}$ and $\Lambda$}

In the definition of the ordering relation we use the notion of the index of a word in $\Lambda^{\prime}$ or in $\Lambda$, denoted $ind(w)$, which is defined to be the highest index of the $t_i^{\prime}$'s, resp. of the $t_i$'s, in the word. Similarly, the index of an element in $\Sigma^{\prime}$ or in $\Sigma$ is defined in the same way by ignoring possible gaps in the indices of the looping generators and by ignoring the braiding part in $\textrm{H}_{n}(q)$. Moreover, the index of a monomial in $\textrm{H}_{n}(q)$ is equal to $0$.

\smallbreak

We are now in position to define the ordering relation on the sets $\Sigma$, $\Sigma^{\prime}$, $\Lambda$ and $\Lambda^{\prime}$:

\begin{defn}[Definition~2, \cite{DL1}] \label{order}
\rm
 We define the following {\it ordering} in the sets $\Sigma_{n}^{\prime}$.

\noindent Let $w={t^{\prime}_{i_1}}^{k_1}{t^{\prime}_{i_2}}^{k_2}\ldots {t^{\prime}_{i_{\mu}}}^{k_{\mu}}$ and $\sigma={t^{\prime}_{j_1}}^{\lambda_1}{t^{\prime}_{j_2}}^{\lambda_2}\ldots {t^{\prime}_{j_{\nu}}}^{\lambda_{\nu}}$, where $k_t , \lambda_s \in \mathbb{Z}$, for all $t,s$. Then:

\smallbreak

\begin{itemize}
\item[(a)] If $\sum_{i=0}^{\mu}k_i < \sum_{i=0}^{\nu}\lambda_i$, then $w<\sigma$.

\vspace{.1in}

\item[(b)] If $\sum_{i=0}^{\mu}k_i = \sum_{i=0}^{\nu}\lambda_i$, then:

\vspace{.1in}

\noindent  (i) if $ind(w)<ind(\sigma)$, then $w<\sigma$,

\vspace{.1in}

\noindent  (ii) if $ind(w)=ind(\sigma)$, then:

\vspace{.1in}

\noindent \ \ \ \ ($\alpha$) if $i_1=j_1, i_2=j_2, \ldots , i_{s-1}=j_{s-1}, i_{s}<j_{s}$, then $w>\sigma$,

\vspace{.1in}

\noindent \ \ \ \ ($\beta$) if $i_t=j_t\ \forall t$ and $k_{\mu}=\lambda_{\mu}, k_{\mu-1}=\lambda_{\mu-1}, \ldots k_{i+1}=\lambda_{i+1}, |k_i|<|\lambda_i|$, then $w<\sigma$,

\vspace{.1in}

\noindent \ \ \ \ ($\gamma$) if $i_t=j_t\ \forall t$ and $k_{\mu}=\lambda_{\mu}, k_{\mu-1}=\lambda_{\mu-1}, \ldots k_{i+1}=\lambda_{i+1}, |k_i|=|\lambda_i|$ and $k_i>\lambda_i$, then $w<\sigma$,

\vspace{.1in}

\noindent \ \ \ \ ($\delta$) if $i_t=j_t\ \forall t$ and $k_i=\lambda_i$, $\forall i$, then $w=\sigma$.

\vspace{.1in}

\item[(c)] In the general case where $w={t^{\prime}_{i_1}}^{k_1}{t^{\prime}_{i_2}}^{k_2}\ldots {t^{\prime}_{i_{\mu}}}^{k_{\mu}} \cdot \beta_1$ and $\sigma={t^{\prime}_{j_1}}^{\lambda_1}{t^{\prime}_{j_2}}^{\lambda_2}\ldots {t^{\prime}_{j_{\nu}}}^{\lambda_{\nu}}\cdot \beta_2$, where $\beta_1, \beta_2 \in \textrm{H}_n(q)$, the ordering is defined in the same way by ignoring the braiding parts $\beta_1, \beta_2$.
\end{itemize}
\end{defn}

\begin{defn} \label{level}
\rm
We define the \textit{subset of level $k$}, $\Lambda_k$, of $\Lambda$ to be the set
$$\Lambda_k:=\{t^{k_0}t_1^{k_1}\ldots t_{m}^{k_m} | \sum_{i=0}^{m}{k_i}=k,\ k_i \in \mathbb{Z}\setminus\{0\},\  k_i \geq k_{i+1}\ \forall i \}.$$
\end{defn}

In \cite{DL1} it was shown that the sets $\Lambda_k$ are totally ordered and well ordered for all $k$ (Proposition~2 \cite{DL1}).

%\subsection{From $\Sigma$ to the augmented set $L$: managing the gaps}
\subsection{Managing the gaps: From bbm's on $\Sigma$ to bbm's on the ${\rm H}_n(q)$-module $L$}

A word in $\Sigma$ is a monomial on $t_i$'s followed by a ``braiding tail'', a monomial on the $g_i$'s. This braiding monomial is a word in the algebra ${\rm H}_n(q)$. The monomial on the $t_i$'s may have gaps in the indices. Using the ordering relation given in Definition~\ref{order} we can manage these gaps by showing that a monomial on $t_i$'s can be expressed as a sum of monomials on $t_i$'s with consecutive indices, which are of less order than the initial word and which are followed by braiding tails (see Theorem~8 \cite{DL1}). We can prove this by only applying conjugation. Note that topologically conjugation corresponds to closing the mixed braid.

\begin{nt}
\rm
For the expressions we obtain after conjugations we use the notation $\widehat{=}$. We also use the symbol $\simeq$ when a stabilization move is performed and $\widehat{\simeq}$ when both stabilization moves and conjugation are performed.
\end{nt}

After managing the gaps in the indices of the $t_i$'s we obtain monomials on $t_i$'s with consecutive indices followed by ``braiding tails''. Note that the exponents of the $t_i$'s are not ordered, so these monomials do not necessarily belong to the basis $\Lambda$ of $\mathcal{S}({\rm ST})$. In order to restrict the bbm's only on elements in $\Lambda$, we need first to augment the set $\Lambda$. So, as a first step we consider the augmented set $L$ in $\mathcal{S}({\rm ST})$ that contains monomials on the $t_i$'s with consecutive indices. In this subsection we consider the set $L$ as an ${\rm H}_n(q)$-module.

\begin{defn}\label{expsetl}
\rm
We define the sets:
$$L^{(n)}\ :=\{t^{k_0}t_1^{k_1}\ldots t_{n}^{k_n},\ k_i \in \mathbb{Z}^*\},\quad L := \bigcup_n L^{(n)},$$ 
\noindent and the \textit{subset of level $k$}, $L_k$, of $L$:
$$L_k:=\{t^{k_0}t_1^{k_1}\ldots t_{m}^{k_m} | \sum_{i=0}^{m}{k_i}=k,\ k_i \in \mathbb{Z}^*\}.$$
\end{defn}

We now show that Eqs~(\ref{eq2}) (Proposition~2) reduce to equations of the same type, but with elements in the set $L$ (i.e. no gaps in the indices).
We need the following lemma, which serve as the basis of the induction applied to prove the main result of this section, Proposition~\ref{imp}.

\begin{lemma}\label{gap}
The equations $X_{\widehat{t_1^{k}}}\ =\ X_{\widehat{t^pt_2^{k}\sigma_1^{\pm 1}}}$ are equivalent to the equations
\[
\begin{array}{llll}
X_{\widehat{{t}^{u_0}{t_1}^{u_1}}} & = & X_{\widehat{t^pt_1^{u_0}t_2^{u_1}\sigma_1^{\pm 1}}}, &\forall\ u_0,u_1<k\ :\ u_0 +u_1\ =\ k, \\
X_{\widehat{t^{k}}} & = & X_{\widehat{t^pt_1^{k}\sigma_1^{\pm 1}}}, & {bbm\ on\ 1st\ strand},\\
X_{\widehat{t^{k}}} & = & X_{\widehat{t^pt_1^{k} \sigma_2\sigma_1^{\pm 1}}\sigma_2^{-1}} & {bbm\ on\ 2nd \ strand}.\\
\end{array}
\]
\end{lemma}

\begin{proof}
We have that:

\[
\begin{array}{ccccc}
t_1^k & = & \underline{t_1^{k-1}\sigma_1}t\sigma_1 & {=} & (q-1)\sum_{j=0}^{k-2}{q^jt^{j+1}t_1^{k-1-j}\sigma_1}\ +\ q^{k-1}\underline{\sigma_1}t^k \sigma_1\\
\downarrow & & & & \downarrow\\
t^{p}t_2^{k}\sigma_1^{\pm 1} & = & t^p\underline{t_2^{k-1}\sigma_2}t_1\sigma_2\sigma_1^{\pm 1} & {=} & (q-1)\sum_{j=0}^{k-2}{q^jt^pt_1^{j+1}t_2^{k-1-j}\sigma_2\sigma_1^{\pm 1}}\ +\ q^{k-1}t^p\underline{\sigma_2}t_1^k \sigma_2\sigma_1^{\pm 1}\\
\end{array}
\]

\noindent Applying now the skein relation we have:
\[
\begin{array}{ccccc}
q^{k-1}\underline{\sigma_1}t^k \sigma_1 & \widehat{=} & q^{k-1}t^k \sigma_1^2 & \overset{skein}{=} & q^{k-1}(q-1)t^k \sigma_1 + q^{k}t^k \\
& & & & \\
q^{k-1}t^p\underline{\sigma_2}t_1^k\sigma_2 \sigma_1^{\pm 1} & \widehat{=} & q^{k-1}t^pt_1^k\sigma_2 \sigma_1^{\pm 1} \underline{\sigma_2} & \overset{skein}{=} &  q^{k}t^pt_1^k\sigma_2 \sigma_1^{\pm 1} \sigma_2^{-1} + q^{k-1}(q-1)t^pt_1^k\sigma_2 \sigma_1^{\pm 1}\\
\end{array}
\]

\noindent and by applying a stabilization move we have:
\[
\begin{array}{rcccc}
q^{k-1}(q-1)t^k \sigma_1 + q^{k}t^k & \simeq & q^{k-1}(q-1)zt^k &+ & q^{k}t^k \\
& & \downarrow && \downarrow\\
q^{k}t^pt_1^k\sigma_2 \sigma_1^{\pm 1} \sigma_2^{-1} + q^{k-1}(q-1)t^pt_1^k\sigma_2 \sigma_1^{\pm 1} & \simeq & q^{k-1}(q-1)zt^pt_1^k\sigma_1^{\pm 1} &+ & q^{k}t^pt_1^k\sigma_2 \sigma_1^{\pm 1} \sigma_2^{-1} \\
\end{array}
\]

\noindent Moreover,

 $(q-1)\sum_{j=0}^{k-2}{q^j t^{j+1}\underline{t_1^{k-1-j}\sigma_1}}\ {=} $\\

 $=\ (q-1)\sum_{j=0}^{k-2}{q^j t^{j+1}\cdot \left[(q-1) \sum_{\phi=0}^{k-2-j}{q^{\phi}t^{\phi}{t_1}^{k-1-j-\phi}}\ + \ q^{k-1-j}\sigma_1 t^{k-1-j}\right]} \ =$\\

 $=\ (q-1)^2\sum_{j=0}^{k-2}\sum_{\phi=0}^{k-2-j}q^{j+\phi}t^{j+1+\phi}t_1^{k-1-j-\phi}\ +\ (q-1)\sum_{j=0}^{k-2}q^{k-1}t^{j+1}t\underline{\sigma_1}t^{k-1-j}\ \simeq$\\

 $\simeq \  (q-1)^2\sum_{j=0}^{k-2}\sum_{\phi=0}^{k-2-j}q^{j+\phi}t^{j+1+\phi}t_1^{k-1-j-\phi}\ +\ (q-1)(k-1)q^{k-1}zt^{k}$,\\

\noindent and

$(q-1)\sum_{j=0}^{k-2}{q^j t^pt_1^{j+1}\underline{t_2^{k-1-j}\sigma_2}\sigma_1^{\pm 1}}\ {=} $\\

 $=\ (q-1)\sum_{j=0}^{k-2}{q^j t^pt_1^{j+1}\cdot \left[(q-1) \sum_{\phi=0}^{k-2-j}{q^{\phi}t_1^{\phi}{t_2}^{k-1-j-\phi}}\ + \ q^{k-1-j}\sigma_2 t_1^{k-1-j}\right]\sigma_1^{\pm 1}} \ =$\\

 $=\ (q-1)^2\sum_{j=0}^{k-2}\sum_{\phi=0}^{k-2-j}q^{j+\phi}t^pt_1^{j+1+\phi}t_2^{k-1-\phi}\sigma_1^{\pm 1}\ +\ (q-1)\sum_{j=0}^{k-2}q^{k-1}t^pt_1^{j+1}t\underline{\sigma_2}t_1^{k-1-j}\sigma_1^{\pm 1}\ \simeq$\\

 $\simeq \  (q-1)^2\sum_{j=0}^{k-2}\sum_{\phi=0}^{k-2-j}q^{j+\phi}t^pt_1^{j+1+\phi}t_2^{k-1-j-\phi}\sigma_1^{\pm 1}\ +\ (q-1)(k-1)q^{k-1}zt^pt_1^{k}\sigma_1^{\pm 1}$.\\

So we have the following:

\[
\begin{array}{lcl}
t_1^{k} & \overset{bbm\ 1^{st}-str.}{\longrightarrow} & t^{p}t_2^{k}\sigma_1^{\pm 1}\\
\widehat{\simeq} & & \widehat{\simeq}\\
{\scriptstyle (q-1)^2\sum_{j=0}^{k-2}\sum_{\phi=0}^{k-2-j}q^{j+\phi}t^{j+1+\phi}t_1^{k-1-j-\phi}} & \overset{bbm\ 1^{st}-str.}{\longrightarrow} & {\scriptstyle (q-1)^2\sum_{j=0}^{k-2}\sum_{\phi=0}^{k-2-j}q^{j+\phi}t^pt_1^{j+1+\phi}t_2^{k-1-j-\phi}\sigma_1^{\pm 1}} \\
 (q-1)(k-1)q^{k-1}zt^{k} & \overset{bbm\ 1^{st}-str.}{\longrightarrow} &  (q-1)(k-1)q^{k-1}zt^pt_1^{k}\sigma_1^{\pm 1}\\
 (q-1)q^{k-1}zt^{k} & \overset{bbm\ 1^{st}-str.}{\longrightarrow} & (q-1)q^{k-1}zt^pt_1^{k}\sigma_1^{\pm 1} \\
 q^kt^k & \overset{bbm\ 2^{nd}-str.}{\longrightarrow} & q^{k}zt^pt_1^{k}\sigma_2\sigma_1^{\pm 1} \sigma_2^{-1}\\
\end{array}
\]
\noindent and this concludes the proof.
\end{proof}

\begin{prop}\label{imp}
In order to obtain an equivalent infinite system to the one obtained from elements in $\Sigma$ by performing bbm on the first moving strand,
it suffices to consider monomials in $L$ followed by braiding tails in ${\rm H}_n(q)$ and perform a braid band move on any strand.
\end{prop}

\begin{proof}
Let $\tau_{gaps}$ be a word containing gaps in the indices but not starting with one. We use Lemma~13 and 14 in \cite{DL1}.
The point is that when managing the gaps, the first part of the word (before the first gap) remains in tact after managing the gaps
and the same carries through after the performance of a braid band move. That is, the following diagram commutes:
\[
\begin{matrix}
\tau\cdot w & \underset{bbm}{\overset{1^{st} str.}{\longrightarrow}} & t^p  \tau_{+} \cdot w_+ g_1^{\pm 1} \\
\mid & & \mid \\
 man. gaps     & & man. gaps\\
\downarrow & & \downarrow \\
\sum_{i} A_i \tau_{i}\cdot w_i & \underset{bbm}{\overset{1^{st} str.}{\longrightarrow}} & \sum_{i} A_i t^p \tau_{i_+}\cdot w_{i_+}g_1^{\pm1} \\
\end{matrix}
\]
\noindent where $\tau\cdot w \in \Sigma$ and $\tau_i \in L, \ \forall i$.

\smallbreak

In the case where the word $\tau\cdot w \in \Sigma$ starts with a gap, we show that equations obtained from $\tau\cdot w$ are equivalent
to equations obtained from elements $\tau_i\cdot w_i \in \Sigma$, where $\tau_i$ are monomials in $t_i$'s {\it not} starting with a gap, {\it but with the bbm
performed on any strand}. We prove this by induction on the strand $m$ where the first gap occurs and the order of $\tau$ in $\Sigma$:

The case $m=1$ is Lemma~\ref{gap}. Suppose that it holds for all elements where the first gap occurs on the $m^{th}$-strand.
 Let $\tau\cdot w = t_{m+1}^k\cdot \alpha$. Then, using Lemma~13 and 14 in \cite{DL1}, for $m+1$ we have:

\[
\begin{array}{ccc}
t_{m+1}^{k}\cdot \alpha & \overset{bbm\ 1^{st}-str.}{\longrightarrow} & t^{p}t_{m+2}^{k}\alpha_{+}\sigma_1^{\pm 1}\\
\widehat{=} & & \widehat{=}\\
(q-1)\sum_{u=0}^{k-1}q^{u-1}t_m^{u}t_{m+1}^{k-u}\alpha\sigma_{m+1} & \overset{bbm\ 1^{st}-str.}{\longrightarrow} & (q-1)\sum_{u=0}^{k-1}q^{u-1}t^pt_{m+1}^{u}t_{m+2}^{k-u}\alpha_{+}\sigma_{m+1}\sigma_1^{\pm 1}\\
q^{k-1}t_m^{k}\underline{\sigma_{m+1}\alpha} \sigma_{m+1} & \overset{bbm\ 1^{st}-str.}{\longrightarrow} & q^{k-1}t^{p}t_{m+1}^{k}\underline{\sigma_{m+2}\alpha_{+}}\sigma_{m+2}\sigma_1^{\pm 1}\\
\end{array}
\]

Interacting now on the left part the braiding generator $\sigma_{m+1}$ with the looping generators in $\alpha$, we obtain words in $\Sigma$
where the first gap occurs on the $m^{th}$-moving strand. We follow the same procedure on the right part and the result follows by the induction hypothesis.
\end{proof}

\begin{figure}
\begin{center}
\includegraphics[width=2.7in]{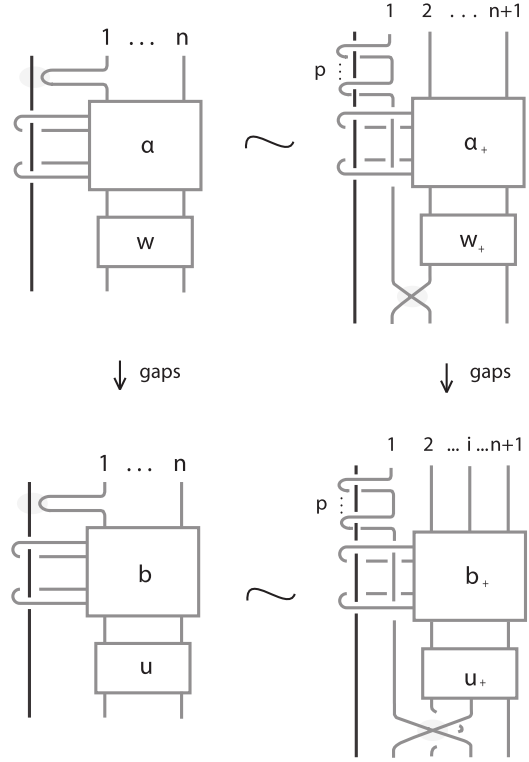}
\end{center}
\caption{Bbm's before and after managing the gaps.}
\label{exgaps}
\end{figure}

\subsection{Ordering the exponents: From bbm's on the ${\rm H}_n(q)$-module $L$ to bbm's on the ${\rm H}_n(q)$-module $\Lambda$}

We now order the exponents and show that equations obtained from elements in the ${\rm H}_n(q)$-module $L$, reduce to equations obtained from elements in the ${\rm H}_n(q)$-module $\Lambda$.

\smallbreak

The monomials in the $t_i$'s that we obtain after managing the gaps are not elements in the set $\Lambda$, since the exponents of the loop generators are not necessarily ordered. In order to express these monomials as sums of elements in the set $\Lambda$ we introduce the following notation:

\begin{nt}
\rm
We set $\tau_{i,i+m}^{k_{i,i+m}}:=t_i^{k_i}t^{k_{i+1}}_{i+1}\ldots t^{k_{i+m}}_{i+m}$, where $m\in \mathbb{N}$ and $k_j\neq 0$ for all $j$. 
\end{nt}

Then we can order the exponents of the $t_i$'s using conjugation. Indeed, we have the following:

\begin{thm}[Theorem~9 \cite{DL1}]\label{exp}
Applying conjugation on an element in $\Sigma$ we have that:

$$\tau_{0,m}^{k_{0,m}}\cdot w\ \widehat{=}\ \sum_{j}{\tau_{0,j}^{\lambda_{0,j}}\cdot w_j},$$

\noindent where $\tau_{0,j}^{\lambda_{0,j}} \in \Lambda$ and $w, w_j \in {\rm H}_n(q),\ \forall j$.
\end{thm}

The procedure we follow in order to prove that equations of the infinite system obtained from elements in $L$ followed by braiding tails in ${\rm H}_{n}(q)$ are equivalent to equations obtained from elements in $\Lambda$ followed by braiding tails, where a braid band move can be performed on any moving strand, is similar to the one described in \cite{DL1}, but, as we mentioned earlier, in this case we do that simultaneously before and after the performance of a braid band move.

\begin{prop}
Equations of the infinite system obtained from elements in $L$ followed by braiding tails in ${\rm H}_{n}(q)$ are equivalent to equations obtained from 
elements in $\Lambda$ followed by braiding tails, where a braid band move can be performed on any moving strand.
\end{prop}

\begin{proof}
It follows from Theorem~9 in \cite{DL1}, since all steps followed so as to order the exponents in a monomial in $t_i$'s, remain
the same after the performance of a bbm.
\end{proof}

\subsection{Eliminating the tails: From the ${\rm H}_{n}(q)$-module $\Lambda$ to $\Lambda$}

We now deal with the braiding tails and prove that equations obtained from elements in $\Lambda$ followed by words in ${\rm H}_n(q)$ by performing bbm's on any moving strand, reduce to equations obtained from elements in $L$ by performing a bbm on any strand.

\smallbreak

In \cite{DL1} we eliminate the braiding `tails' by applying conjugation and stabilization moves, denoted by $\widehat{\simeq}$. We have the following:

\begin{thm}[Theorem~10 \cite{DL1}] \label{tails}
Applying conjugation and stabilization moves on a word in the $\bigcup_{\infty}{\rm H}_n(q)$-module, $\Lambda$ we have that:

$$\tau_{0,m}^{k_{0,m}}\cdot w_n\ \widehat{\simeq}\  \sum_{j}{f_j(q,z)\cdot \tau_{0,u_j}^{v_{0,u_j}}},$$

\noindent such that $\sum{v_{0,u_j}}=\sum{k_{0,m}}$ and $\tau_{0,u_j}^{v_{0,u_j}} < \tau_{0,m}^{k_{0,m}}$, for all $j$.
\end{thm}

Applying the same technique before and after the performance of a bbm, we have the following result:

\begin{prop}\label{tailbbm}
Equations of the infinite system obtained from elements in $\Lambda$ followed by words in ${\rm H}_n(q)$ are equivalent to equations
 obtained from elements in $L$ by performing a braid band move on any moving strand.
\end{prop}

\begin{proof}
We perform a bbm and we cable the parallel strand with the surgery strand. We then apply Theorem~10 \cite{DL1} before and after
the performance of the bbm and uncable the parallel strand. The proof is illustrated in Figure~\ref{tailfig}.
\end{proof}

\begin{figure}[\!ht]
\begin{center}
\includegraphics[width=5.3in]{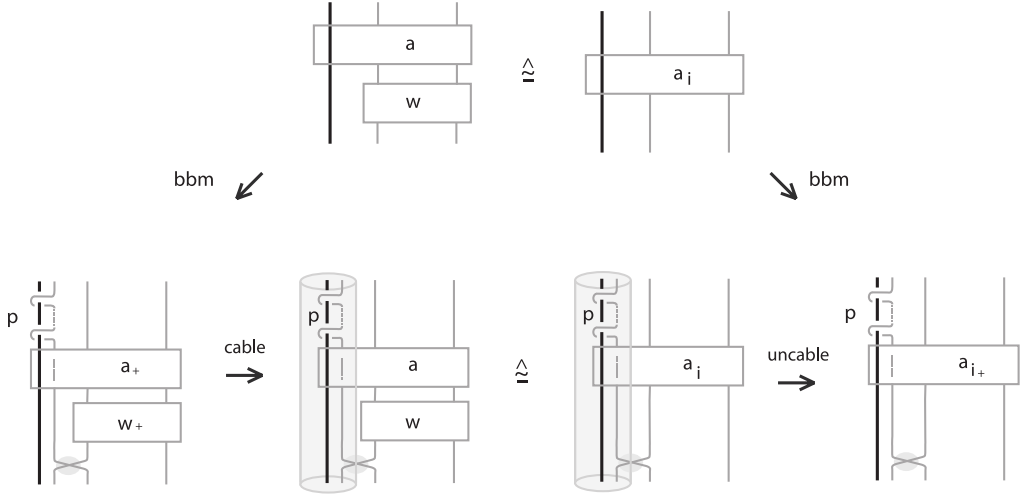}
\end{center}
\caption{ The proof of Proposition~\ref{tailbbm}. }
\label{tailfig}
\end{figure}

\begin{ex}
In this example we demonstrate Proposition~\ref{tailbbm}.
\[
\begin{matrix}
tt_1t_2\cdot g_1g_2g_1 & \underset{bbm}{\overset{1^{st} str.}{\longrightarrow}} & t^p t_1t_2t_3 \cdot g_2g_3g_2 g_1^{\pm 1} \\
\mid & & \mid \\
elim. tails     & & elim. tails\\
\downarrow & & \downarrow \\
(q-1)(q^2-q+1)\cdot tt_1t_2 & \underset{bbm}{\overset{1^{st} str.}{\longrightarrow}} & (q-1)(q^2-q+1)\cdot t^pt_1t_2t_3 g_1^{\pm1} \\
+ & & + \\
q(q-1)^2z \cdot tt_1^2 & \underset{bbm}{\overset{1^{st} str.}{\longrightarrow}} & q(q-1)^2z \cdot t^pt_1t_2^2 g_1^{\pm1}\\
+ & & + \\
a\cdot t^2t_1 & \underset{bbm}{\overset{1^{st} str.}{\longrightarrow}} & a\cdot t^pt_1^2t_2g_1^{\pm1}\\
+ & & + \\
[q^2(q-1)(q^2-q+1)z^2]\cdot t^3 & \underset{bbm}{\overset{1^{st} str.}{\longrightarrow}} & [q^2(q-1)(q^2-q+1)z^2]\cdot t^pt_1^3g_1^{\pm1}\\
\end{matrix}
\]

\noindent where $a= q^3z+q^2(q-1)^2+2q^2(q-1)^2z+q(q-1)^4z$.
\end{ex}

Following the same procedure and applying the same techniques as in \cite{DL1}, we obtain the following:

\begin{thm}
It suffices to perform braid band moves on any strand on elements in the basis $\Lambda$ of $\mathcal{S}({\rm ST})$, in order to obtain the equations needed for computing the Homflypt skein module of the lens spaces $L(p,1)$.
\end{thm}

\begin{proof}
The proof is based on Theorems~8, 9 and 10 from \cite{DL1} and the fact that the braid band moves commute with the stabilization moves
and the skein (quadratic) relation (Lemma~\ref{bbm1} in this paper). The fact that the braid band moves and conjugation do not commute, results in the need of performing braid band moves on all moving strands of the elements in $\Lambda$.
\end{proof}

\begin{figure}
\begin{center}
\includegraphics[width=2.5in]{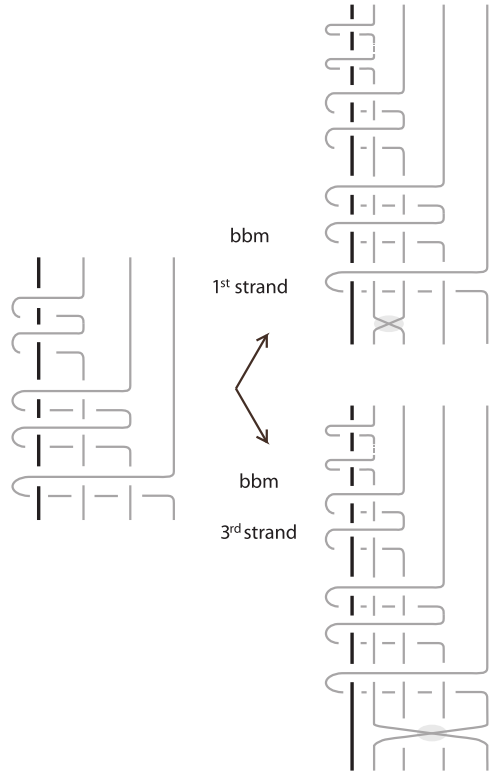}
\end{center}
\caption{The performance of a bbm on the $1^{st}$ and on the $3^{rd}$ strand of an element in $\Lambda$.}
\label{exlambda}
\end{figure}

To summarize the above results, we have that the computation of the Homflypt skein module of the lens spaces $L(p,1)$ reduces to:

\begin{itemize}
\item[(i)] considering elements in $\Lambda$ and
\smallbreak
\item[(ii)] performing bbm's on any strand.
\end{itemize}

In particular, in order to compute $\mathcal{S}(L(p,1))$, it suffices to solve the infinite system of equations:

$$ X_{\widehat{\tau}}\ =\ X_{\widehat{sl_i(\tau)}},$$

\noindent where $sl_i(\tau)$ is the result of the performance of bbm on the $i^{th}$-moving strand of $\tau \in \Lambda$, for all $\tau \in \Lambda$ and for all $i$.

\section{Conclusions}

In this paper we related $\mathcal{S}(L(p,1))$ to $\mathcal{S}({\rm ST})$ and in particular we showed that in order to compute the Homflypt skein module of the lens spaces $L(p,1)$, we need to solve an infinite system resulting from the performance of braid band moves on any strand on elements in the basis $\Lambda$ of $\mathcal{S}({\rm ST})$. This is a very technical and difficult task and is the subject of a sequel paper.

\end{document}